\documentclass[11pt]{amsart}
\usepackage{color, amssymb, amsmath}
\usepackage{enumitem}

\usepackage[colorlinks]{hyperref}
\hypersetup{linkcolor=blue, urlcolor=blue, citecolor=red}
\usepackage[capitalise]{cleveref}

\newtheorem{theorem}{Theorem}[section]
\newtheorem{proposition}[theorem]{Proposition}
\newtheorem{corollary}[theorem]{Corollary}
\newtheorem{lemma}[theorem]{Lemma}
\newtheorem{question}[theorem]{Question}

\theoremstyle{definition}
\newtheorem{definition}[theorem]{Definition}

\newcommand{\N}{\mathbb N}
\newcommand{\Tau}{\mathcal T}

\newcommand{\im}{\operatorname{im}}
\newcommand{\dom}{\operatorname{dom}}
\newcommand{\id}{\operatorname{id}}
\newcommand{\set}[2]{\left\{#1:#2\right\}}

\title{Topological embeddings into transformation monoids}

\author{S. Bardyla, L. Elliott, J. D. Mitchell and Y. P\'eresse}

\address{S.~Bardyla: Institute of Mathematics,
P.J. \v{S}af\'arik University in Ko\v sice, Slovakia; and Institute of Discrete Mathematics and Geometry, TU Wien, Austria.}
\email{sbardyla@gmail.com}
\address{L. Elliott: Department of Mathematics and Statistics Binghamton
  University, PO Box 6000, Binghamton, New York 13902-6000, USA}
\email{luke.elliott142857@gmail.com}
\address{J. D. Mitchell: University of St Andrews, School of Mathematics and
Statistics, Scotland, UK}
\email{jdm3@st-andrews.ac.uk}
\address{Y. P\'eresse: University of Hertfordshire, Department of Physics,
Astronomy and Mathematics, Hatfield, Hertfordshire, UK}
\email{y.peresse@herts.ac.uk}
\makeatletter
\@namedef{subjclassname@2020}{%
  \textup{2020} Mathematics Subject Classification}
\makeatother


\thanks{The first named author was supported by the Slovak Research and Development Agency under the contract No. APVV-21-0468 and by the Austrian Science Fund FWF (Grant I 5930).}

\begin{document}
\begin{abstract}
In this paper we consider the questions of which topological semigroups
    embed topologically into the full transformation monoid $\N ^ \N$ or the
    symmetric inverse monoid $I_{\N}$ with their respective canonical Polish
    semigroup topologies.
 We characterise those topological semigroups that embed topologically into
    $\N ^ \N$ and belong to any of the following classes: commutative semigroups;
    compact semigroups; groups; and certain Clifford semigroups. We prove
    analogous characterisations for topological inverse
    semigroups and $I_{\N}$. We construct several examples of countable Polish
    topological semigroups that do not embed into $\N ^ \N$, which answer, in
    the negative, a recent open problem of Elliott et al. Additionally, we
    obtain two sufficient conditions for a topological Clifford semigroup to be
    metrizable, and prove that inversion is automatically continuous in every
    Clifford subsemigroup of $\N^\N$. The former complements recent works of
    Banakh et al.
\end{abstract}
\maketitle

\section{Introduction}

As is well-known, Cayley's Theorem states that every group is isomorphic to a subgroup of a symmetric
group $S_X$ on some set $X$. Analogous statements hold for semigroups and inverse semigroups. More specifically, 
every semigroup is isomorphic to a subsemigroup of the 
monoid $X ^ X$ consisting of all transformations of some set $X$ with operation the usual composition of functions; for more details see \cite[Theorem 1.1.2]{Howie1995aa}. 
Similarly, every inverse semigroup is isomorphic to an inverse subsemigroup of
a symmetric inverse monoid $I_X$, consisting of all partial permutations of the set $X$ with operation the usual composition of binary relations; see \cite[Theorem 5.1.7]{Howie1995aa}. 
Monoids of transformations and partial permutations have
been extensively studied in the literature; of particular relevance to this
paper are~\cite{Bor, EMP, EJMPP, Gan, Kud, Law, Dragan,  Mesyan1, Mesyan2, 
MMMP, MS2023,PP, PP1, Per, PS}.

It seems natural enough to ask if there is an analogue of Cayley's Theorem for semigroups endowed with topologies that are compatible with their algebraic structures. 
The monoids $\N^\N$, $S_{\N}$, and $I_\N$ each possess a canonical topology with respect
to which their operations are continuous. 
As a topological space, $\N ^ \N$ with the Tychonoff product topology arising from the discrete topology on
$\N$, is the well-known \textit{Baire space}.
The space $\N ^ \N$ is Polish, \textit{i.e.} completely
metrizable and separable; for further details see~\cite[Section 3]{Kechris1995}.
Multiplication as a function from $\N^{\N}\times \N^{\N}$ (with the
product topology) to $\N^{\N}$ is continuous, and as such $\N^{\N}$ is also a
\textit{topological semigroup}.
It was shown in \cite[Theorem 5.4(b)]{main} that
the unique Polish semigroup topology on $\N^{\N}$ is the topology of the Baire space. We refer to this topology as the canonical topology for $\N ^ \N$.

Topological groups and inverse semigroups are defined analogously, where both multiplication and inversion are continuous.
Since $S_{\N}$ is a $G_{\delta}$ subset of the Baire space $\N ^ \N$, the subspace topology on $S_{\N}$ is Polish.
Gaughan~\cite{Gaughan} showed that this topology is contained in every Hausdorff group topology on $S_{\N}$.  
On the other hand, every Borel measurable bijection between Polish groups is a homeomorphism (\cite[Propositions 3.2 and 3.3]{main}). Hence if $G$ is a Polish topological group with respect to topologies $\Tau_1$  and $\Tau_2$ where $\Tau_1\subseteq \Tau_2$, then the identity function from $(G, \Tau_2)$ to $(G, \Tau_1)$ is continuous, and hence Borel measurable, which implies that $\Tau_1 = \Tau_2$.
It follows that the subspace topology induced by the canonical topology on $\N ^ {\N}$ is the unique Polish group topology on $S_{\N}$. As such we refer to the subspace topology on $S_{\N}$ induced by the canonical topology on $\N ^ \N$ as the canonical topology for $S_{\N}$. 
Although not a subspace of $\N ^ \N$,
the symmetric
inverse monoid $I_{\N}$ also possesses a natural unique Polish inverse semigroup
topology~\cite[Theorem 5.15(ix)]{main}; a subbasis for this topology is given in \eqref{eq-subbasis-IN}.
The
symmetric group $S_{\N}$ is a subgroup of both $I_{\N}$ and $\N^\N$ and the
canonical topologies on these semigroups both induce the canonical topology on
$S_{\N}$. 
The
canonical topologies on $S_\N$, $\N^\N$ and $I_\N$ are zero-dimensional, i.e.
they each possess a basis consisting of clopen sets (see \eqref{eq-subbasis-NN} and \eqref{eq-subbasis-IN} for further details).
 Since all three spaces are Polish, zero-dimensional, and their compact subsets have empty interior, the Alexandrov-Urysohn Theorem (\cite[Theorem 7.7]{Kechris1995}) implies that $S_{\N}, I_{\N}$ and $\N ^ \N$ are homeomorphic, although they are clearly not isomorphic monoids. 

Throughout the remainder of this paper, unless explicitly stated  otherwise, we
write $S_\N$, $\N^\N$, and $I_\N$ to mean the corresponding topological
group, monoid, and inverse monoid endowed with their canonical topologies.

In this paper we consider the problems of which topological semigroups can be topologically embedded into $\N^{\N}$, and which topological inverse semigroups 
embed topologically into $I_{\N}$. It is well-known that a Hausdorff topological group $G$ embeds
topologically into the symmetric group $S_\N$ if and only if $G$ is
second-countable and possesses a neighbourhood basis of the identity consisting
of open subgroups; see, for example, \cite[Theorem 5.1]{Bor}.  This characterisation can be readily applied to an arbitrary topological group $G$ to determine whether or not $G$ embeds topologically into the symmetric group. 
For example, the additive group $\mathbb Q$ of rational numbers endowed with the subspace topology inherited from the real line, despite being second-countable and zero-dimensional, cannot be embedded into $S_{\N}$, because $(-1,1)\cap \mathbb Q$ is an open neighborhood of $0$ which contains no open subgroup of $\mathbb Q$.   

There are natural analogues, in the contexts
of semigroups and inverse semigroups, of the aforementioned characterisation for groups, in terms of right congruences; see Propositions~\ref{cool} and~\ref{cool1}.  
Right congruences of topological semigroups are, in general, much more complicated, and harder to work with than subgroups of topological groups. 
As such, unlike in the case of groups, Propositions~\ref{cool} and~\ref{cool1}
do not often simplify the process of 
determining whether or not a topological semigroup is topologically isomorphic
to a subsemigroup of $\N ^ \N$ or $I_{\N}$. 
Every countable Polish group is discrete, and so the trivial group is a neighbourhood basis of the identity; that is, every countable Polish group embeds topologically into $S_\N$, and hence into $\N ^ \N$ and $I_{\N}$ also.
On the other hand, the situation for countable Polish semigroups that are not groups is unclear; the following question was posed in \cite{main}.

\begin{question}[cf. Question 5.6 in \cite{main}]\label{question-main}
  Does every countable Polish semigroup embed topologically into
  $\N^{\N}$?
\end{question}

In this paper, we provide several alternate characterisations of particular types of topological semigroup that embed into $\N ^ \N$ or $I_{\N}$.

The paper is organised as follows. In \cref{section-statement} we state the main results of the paper; in \cref{section-proofs} we prove the main theorems; and in \cref{section-counter-examples} we provide a number of examples demonstrating the sharpness of our main results and 
we show that the answer to \cref{question-main} is negative. 


\section{Statement of main results}\label{section-statement}

We start by giving some preliminary material required to state the main results of this paper. 
Recall that, a semigroup $S$ is called an {\em inverse semigroup} if for each element $x\in
S$ there exists a unique inverse element $x^{-1}\in S$ such that $xx^{-1}x=x$
and $x^{-1}xx^{-1}=x^{-1}$.
An inverse semigroup $S$ is called {\em Clifford} if $xx ^ {-1} = x ^
{-1}x$ for all $x\in S$; or, equivalently, $S$ is a strong semilattice of
groups. 
Recall that a {\em semilattice} is a commutative semigroup of
idempotents. For every inverse semigroup $S$ the set $E(S)=\set{e\in S}{e ^ 2=e}$
is a semilattice. 
For a partial function $f$ on $X$ we denote the domain and image of $f$ by $\dom(f)$ and $\im(f)$, respectively. 
Throughout this paper we will write partial functions to the right of their arguments and compose from left to right; it will also be convenient to identify  partial functions $f : X \to X$ with their graphs $\set{(x, (x)f)}{x\in \dom(f)}\subseteq X \times X$.
Elements
of $I_X$ are referred to as \textit{partial permutations of the set $X$}.

A subbases for the canonical topologies on $\N ^ \N$ consists of the sets:
\begin{equation}\label{eq-subbasis-NN}
U_{x,y}=\set{f\in \N ^ {\N}}{(x,y)\in f}.
\end{equation}
The complement of $U_{x,y}$ is $\bigcup_{z\neq y} U_{x, z}$, which shows that the subbasic open sets are clopen, and hence $\N ^ \N$ with its canonical topology is zero-dimensional.
The following sets form a subbasis for the canonical topology on $I_\N$:
\begin{equation}\label{eq-subbasis-IN}
    U_{x, y}= \set{h\in I_{\N}}{(x,y)\in h},\hbox{ } W_x = \set{h\in I_{\N}}{ x\notin
    \dom(h)},\hbox{ } W_x^ {-1} =\set{h\in I_{\N}}{x\notin \im(h)},
\end{equation}
where $x,y\in\N$. 

 An {\em embedding} of a semigroup $S$ into a semigroup $T$ is an
injective homomorphism from $S$ to $T$. For topological semigroups $S$ and $T$ an embedding
$\phi: S \to T$ is called {\em topological} if both of the maps $\phi$ and
$\phi ^ {-1}: (S)\phi \to S$ are continuous. The term \textit{topological
isomorphism} is defined analogously.

In this paper we characterize the commutative topological subsemigroups of
$\N^{\N}$ and $I_\N$ as follows.

\begin{theorem}\label{I am done inventing new labels}
    A commutative topological semigroup $S$ embeds topologically into $\N^\N$
    if and only if there exists a countable family $\set{S_n}{n\in\N}$ of
    countable discrete semigroups such that $S$ embeds topologically into the
    Tychonoff product $\prod_{n\in\N}S_n$.
\end{theorem}

A semigroup $S$ with adjoined external zero is denoted by $S^0$.

\begin{theorem}\label{I am done inventing new labels1}
    A commutative topological inverse semigroup $S$ embeds topologically into
    $I_\N$ if and only if there exists a countable family $\set{G_n}{n\in\N}$ of
    countable groups such that $S$ embeds topologically into the Tychonoff
    product $\prod_{n\in\N}G^0_n$, where each factor is discrete.
\end{theorem}

A topological space $X$ is called {\em totally disconnected} if the only
connected subsets of $X$ are singletons. It is well-known that every Tychonoff
zero-dimensional space is totally disconnected, but there exists a Polish totally
disconnected topological group which is not zero-dimensional~\cite[Proposition
 4.3]{Vmill}. However these two notions coincide for subspaces of the real line.
It is also well-known, see \cite[Theorem 7.8]{Kechris1995} for example, that a topological space $X$ is homeomorphic to a
subspace of $\N^\N$ if and only if $X$ is metrizable, zero-dimensional, and
second-countable. In this paper we obtain the following characterization of the
compact subsemigroups of $\N^\N$.

\begin{theorem}\label{theorem-compact}
    Let $S$ be a compact topological semigroup. Then the following are equivalent:
    \begin{enumerate}[label=\rm (\roman*)]
    \item $S$ is homeomorphic to a subspace of $\N ^ \N$ (and $I_{\N})$;
    \item $S$ embeds topologically into $\N^\N$;
    \item $S$ is metrizable and totally disconnected.
    \end{enumerate}
\end{theorem}

The compact inverse subsemigroups of $I_{\N}$ are characterized as follows.

\begin{theorem}\label{theorem-compact-2}
    Let $S$ be a compact inverse topological semigroup. Then the following are
    equivalent:
    \begin{enumerate}[label=\rm (\roman*)]
        \item $S$ is homeomorphic to a subspace of $\N ^ \N$ (and $I_{\N}$);
        \item $S$ embeds topologically into $I_{\N}$;
        \item $S$ embeds topologically into $\N^\N$;
        \item $S$ is metrizable and totally disconnected.
    \end{enumerate}
\end{theorem}

Turning to topological groups, we obtain the following characterisation.

\begin{theorem}\label{theorem-groups}
    Let $G$ be a Hausdorff topological group. Then the following conditions are
    equivalent:
    \begin{enumerate}[label=\rm (\roman*)]
    \item $G$ embeds topologically into $S_{\N}$;
      \item $G$ embeds topologically into $I_\N$;
      \item $G$ embeds topologically into $\N^\N$;
      \item $G$ is second-countable and has a neighbourhood basis of the
        identity consisting of open subgroups.
    \end{enumerate}
\end{theorem}

Perhaps the next most natural objective is to characterize those countable Polish
Clifford semigroups that topologically embed into $\N^\N$ or $I_\N$. 
The following notion introduced by Banakh and Pastukhova~\cite{BP}
is crucial for this purpose.

\begin{definition}\label{def}
    An inverse semigroup $X$ endowed with a semigroup topology is called {\em ditopological} if  inversion is continuous; and for
    any point $x\in X$ and any neighborhood $O$ of $x$ there are neighborhoods
    $U$ and $W$ of $x$ and $xx^{-1}$, respectively, such that
    $$\set{s\in S}{\exists{b}\in U,\ \exists e\in W\cap E(S) \text{ such that
    }b=es}\cap \set{s\in S}{ss^{-1}\in W}\subseteq O.$$
\end{definition}

The countable Polish Clifford subsemigroups of $I_\N$ are characterised as follows.

\begin{theorem}\label{emb22}
    A countable Polish Clifford semigroup $S$ embeds topologically into $I_\N$
    if and only if $S$ is ditopological and the semilattice $E(S)$ embeds
    topologically into $I_\N$.
\end{theorem}

There exist countable commutative Polish Clifford semigroups with compact semilattice of idempotents that are not ditopological; see \cref{exB}. 

Automatic continuity of inversion in paratopological groups, or more
general, inverse topological semigroups was investigated by many authors
in~\cite{BG, BR, Brand, Ellis, GR, Mon, Pfi, Roma, Tka}. 
The following theorem implies that inversion is 
automatically continuous in every Clifford subsemigroups of $\N ^ \N$.

\begin{theorem}\label{ditop}
Each Clifford subsemigroup of $\N^\N$ is ditopological.    
\end{theorem}

Also, \cref{ditop} allows us to characterize countable Polish Clifford subsemigroups of
$\N^\N$.

\begin{theorem}\label{emb23}
    A countable Polish Clifford semigroup $S$ embeds topologically into $\N^\N$
    if and only if $S$ is ditopological and the semilattice $E(S)$ embeds
    topologically into $\N^\N$.
\end{theorem}

Given \cref{emb22} and \cref{emb23} it is natural to ask if the semilattice $E(S)$ embeds into $\N^\N$ if and only if it embeds into $I_{\N}$ when $S$ is any countable Polish Clifford semigroup.
We will show in \cref{embcl} that every subsemilattice of $I_{\N}$ embeds topologically into $\N^\N$. The converse implication, however, fails, see \cref{chain-finite} for a counter-example. Furthermore, in~\cref{locally_compact_ditosemiex} we give an example of a countable Polish linear semilattice which is not topologically
isomorphic to a subsemigroup of either $\N^\N$ or $I_\N$.

For an element $x$ of a semilattice $X$ let ${\uparrow}x=\set{y\in X}{x\leq y}$.
\begin{definition}\label{def semilat}
    A semilattice $X$ endowed with a topology is called a:
    \begin{enumerate}[label=\rm (\roman*)]
      \item {\em $U$-semilattice}, if for every open set $U$ and every $x\in U$ there exist $y\in U$ and an open neighborhood $V$ of $x$ such
        that $V\subseteq{\uparrow} y$;
      \item {\em $U_2$-semilattice}, if for every open set $U$ and every $x\in U$ there exist $y\in U$ and a clopen ideal $I\subseteq X$ such that $x\in X\setminus I\subseteq {\uparrow} y$.
    \end{enumerate}
\end{definition}
Clearly, any $U_2$-semilattice is a $U$-semilattice, but the converse implication fails (see~\cite{BP} and
\cite{CHK} for more details).

Theorems~\ref{emb22} and~\ref{emb23} are derived from the following
more general theorems.

\begin{theorem}\label{theorem-we-need-better-labels}
    Let $S$ be a  Clifford topological semigroup whose set of idempotents
    $E(S)$ is a $U$-semilattice. Then $S$ embeds topologically into $I_{\N}$ if
    and only if $S$ is Hausdorff, ditopological and every maximal subgroup of
    $S$, as well as the semilattice $E(S)$,  embed topologically into $I_\N$.
\end{theorem}

\begin{theorem}\label{newtheorem1}
    Let $S$ be a Clifford topological semigroup whose set of idempotents $E(S)$
    is a $U_2$-semilattice. Then $S$ embeds topologically into $\N^{\N}$ if and
    only if $S$ is Hausdorff, ditopological and every maximal subgroup of $S$,
    as well as the semilattice $E(S)$,  embed topologically into $\N^\N$.
\end{theorem}

Theorems~\ref{theorem-we-need-better-labels} and~\ref{newtheorem1} imply the
following corollary which complements results of Banakh et al. about the
metrizability of Clifford topological inverse semigroups~\cite{B,Bok,BGPR}.

\begin{corollary}
    A topological Clifford semigroup $S$ is metrizable and zero-dimensional if
    one of the following conditions holds:
    \begin{enumerate}[label=\rm (\roman*)]
        \item $S$ is Hausdorff, ditopological, $E(S)$ is a $U_2$-semilattice
          which embeds topologically into $\N^\N$ and each maximal subgroup of
          $S$ embeds topologically into $\N^\N$;
        \item $S$ is Hausdorff, ditopological, $E(S)$ is a $U$-semilattice
          which embeds topologically into $I_\N$ and each maximal subgroup of
          $S$ embeds topologically into $I_\N$.
    \end{enumerate}
\end{corollary}

\section{Proofs of the main theorems}\label{section-proofs}

In this section we prove the main results of this paper each in its own
subsection. We begin by recording some results that are useful throughout this
section.

An equivalence relation $\rho$ on a
semigroup $S$ is called a {\em right congruence} if for any $x,y,z\in S$,
$(x,y)\in\rho$ implies $(xz,yz)\in\rho$. If $\rho$ is a right congruence on a
semigroup $S$ and $x\in S$, then the set $\set{y\in S}{(x,y)\in\rho}$ is denoted by $[x]_{\rho}$.

\begin{proposition}[cf. Theorem~5.5 \cite{main}] \label{cool}
    A Hausdorff topological semigroup $S$ is topologically isomorphic to a
    subsemigroup of $\N^{\N}$ if and only if there exists a countable family
    $\set{\rho_i}{i\in\N}$ of right congruences of $S$, each having countably many
    equivalence classes, such that the family $\set{[x]_{\rho_i}}{x\in S, i\in\N}$
    is a basis for the topology on $S$.
\end{proposition}

A right congruence $\rho$ on an inverse monoid $S$ is called {\em
Vagner-Preston} if for every $s\in S$ either: $t\in [s]_{\rho}$ implies that
$1\in [tt^{-1}]_{\rho}$; or $[st]_{\rho}=[s]_{\rho}$ for all $t\in S$.  The
following analogue of \cref{cool} for the symmetric inverse monoid $I_\N$ was
proven in~\cite[Theorem 5.21]{main}.

\begin{proposition}\label{cool1}
  A Hausdorff topological inverse monoid $S$ is topologically isomorphic to an inverse
  subsemigroup of $I_\N$ if and only if there exists a sequence
  $\set{\rho_i}{i\in\N}$ of Vagner-Preston right congruences, each having
  countably many equivalence classes, such that the family
  $\set{[s]_{\rho_i},[s]_{\rho_i}^{-1}}{s\in S, i\in \N}$ is a subbasis\footnote{The term ``subbasis'' cannot be replaced with ``basis'' here, unlike in \cref{cool}.} 
  of the topology on $S$.
\end{proposition}

As we mentioned in the introduction, for any countable semigroup $S$ the
diagonal congruence $\Delta_S=\set{(x,x)}{x\in S}$ has countably many equivalence
classes and the family $\set{[x]_{\Delta_S}}{x\in S}$ forms a basis for the
discrete topology on $S$. Combined with \cref{cool} this observation yields the
following corollary.

\begin{corollary}\label{disc}
    Each countable discrete semigroup embeds topologically into $\N^\N$.
\end{corollary}

The following lemma will be useful for detecting semigroups topologically
embeddable into $\N^\N$.

\begin{lemma}\label{prod}
    Suppose that $S_n$ is a topological semigroup for every $n\in \N$. Then the following hold:
    \begin{enumerate}[label=\rm(\roman*)]
        \item If $S_n$ is topologically isomorphic to a subsemigroup of $\N^{\N}$ for every
              $n\in \N$, then the Tychonoff product $\prod_{n\in\N}S_n$ embeds topologically
             into $\N^{\N}$.
        \item If $S_n$ is topologically isomorphic to a subsemigroup of $I_{\N}$ for every
              $n\in\N$, then the Tychonoff product $\prod_{n\in\N}S_n$ embeds topologically
              into $I_{\N}$.
    \end{enumerate}
\end{lemma}
\begin{proof}
    We only prove part (ii), the proof of item (i) is similar. We will show
    that the Tychonoff power $I_\N^{\N}$ embeds topologically into $I_\N$. Fix
    any partition of $\N$ into countably many infinite subsets $A_i$ for
    $i\in\N$. Consider the subsemigroup $Y=\set{f\in I_\N}{(A_i)f\subseteq A_i\text{
    for all }i\in \N}$ of $I_\N$. For each $f\in Y$ let
    $f_n=f{\restriction}_{A_n}$. Then it is straightforward to check that the
    map $\phi: Y\rightarrow \prod_{i\in \N }I_{A_i}\cong I_{\N}^\N$ defined by
    $(f)\phi=(f_n)_{n\in\N}$ is the desired topological isomorphism.
\end{proof}

A semigroup $S$ is called {\em countably prodiscrete} if $S$ can be embedded
into the Tychonoff product of countably many countable discrete semigroups.
\cref{disc} and~\cref{prod} imply the following.

\begin{corollary}\label{cnew}
    Every countably prodiscrete semigroup can be topologically embedded into $\N^{\N}$.
\end{corollary}

\subsection{Proof of \cref{I am done inventing new labels}}

We need to show that a commutative topological semigroup $S$ embeds
topologically into $\N^\N$ if and only if $S$ is countably prodiscrete.
\begin{proof}
($\Leftarrow$)
  By~\cref{cnew} each countably prodiscrete semigroup embeds topologically into
  $\N^\N$. 
  
($\Rightarrow$)
  Fix any commutative subsemigroup $S$ of $\N^\N$. By~\cref{cool}, the
  semigroup $S$ possesses a countable family $\set{\rho_n}{n\in\N}$ of right
  congruences of $S$, each having countably many classes, such that the family
  $\set{[x]_{\rho_n}}{x\in S, n\in\N}$ is a basis of the subspace topology on $S$
  inherited from $\N^\N$. Since $S$ is commutative, for each $n\in\N$, $\rho_n$
  is a two-sided congruence. Clearly, for every $n\in\N$ the quotient semigroup
  $S_n=S/{\rho_n}$ is countable. We will show that $S$ embeds topologically
  into the Tychonoff product $\prod_{n\in\N}S_n$, where each factor is endowed
  with the discrete topology. For every $n\in\N$ let $f_n:S\rightarrow S_n$ be
  the homomorphism associated with the congruence $\rho_n$. Since each
  equivalence class $[x]_{\rho_n}$ is clopen, the homomorphism $f_n$ (onto the
  discrete semigroup $S_n$) is continuous. It follows that the diagonal map
  $\Phi:S\rightarrow \prod_{n\in\N}S_n$, $(x)\Phi=((x)f_n)_{n\in\N}$ is
  continuous as well. Since $\N ^ \N$ is Hausdorff, $S$ is Hausdorff and so 
  $\bigcap_{n\in\N}[x]_{\rho_n}=\{x\}$ for every $x\in S$. Thus,  $\Phi$ is
  injective. To show that $\Phi$ is a topological embedding, fix any open basic
  set $[x]_{\rho_n}$ of $S$. The image $$([x]_{\rho_n})\Phi=\set{(z_i)_{i\in\N}\in
  (S)\Phi}{z_n=(x)f_n}$$ is open in $(S)\Phi$. Hence the map $\Phi$ is open
  onto its image, which implies that $\Phi$ is a topological embedding.
\end{proof}

Recall that for a semigroup $S$ by $S^1$ and $S^0$ we denote the semigroup $S$
with adjoined external identity and zero, respectively. If $S$ is a topological
semigroup, then, unless otherwise stated explicitly, we write $S^1$ to mean the topological semigroup obtained from $S$ by extending the topology of $S$ and where the adjoined identity is isolated. An analogous statement holds
for $S ^ 0$. In order to prove \cref{I am done inventing new labels1} we need
the following lemmas.


\begin{lemma}\label{usefulnew}
    For a topological semigroup $S$ the following assertions are equivalent:
    \begin{enumerate}[label=\rm(\roman*)]
        \item $S$ embeds topologically into $\N^\N$;
        \item $S^1$ embeds topologically into $\N^\N$;
        \item $S^0$ embeds topologically into $\N^\N$.
    \end{enumerate}
\end{lemma}

\begin{proof}
    The implications (ii) $\Rightarrow$ (i) and (iii) $\Rightarrow$ (i) are
    trivial. Let $S$ be a topological semigroup which embeds topologically into
    $\N^\N$. Put $2\N=\set{2n}{n\in\N}$. Since $2\N^{2\N}$ is topologically
    isomorphic to $\N^\N$ we obtain that $S$ is topologically isomorphic to a
    subsemigroup $T$ of ${2\N}^{2\N}$. For each element $t\in T$ let $$t'=t\cup
    \{(1,1)\} \cup \set{(2n+1,3)}{n\in\N\setminus\{0\}}.$$ Clearly, $T'=\set{t'}{t\in T}$ is a subsemigroup of $\N^\N$. It is easy to check that $S^1$ is
    topologically isomorphic to a subsemigroup $T'\cup\{\id_\N\}$
    of $\N^\N$, where by $\id_\N$ we denote the identity
    permutation of $\N$. Hence the implication (i) $\Rightarrow$ (ii) holds.
    Also, it is easy to check that $S^0$ is topologically isomorphic to a
    subsemigroup $T'\cup\{z\}$ of $\N^\N$, where $z=\set{(n,1)}{n\in\N}$. Hence
    the implication (i) $\Rightarrow$ (iii) holds.
\end{proof}

Similarly one can prove the following lemma.
\begin{lemma}\label{usefulnew1}
    For a topological semigroup $S$ the following assertions are equivalent:
    \begin{enumerate}[label=\rm(\roman*)]
        \item $S$ embeds topologically into $I_\N$;
        \item $S^1$ embeds topologically into $I_\N$;
        \item $S^0$ embeds topologically into $I_\N$.
    \end{enumerate}
\end{lemma}

\begin{lemma}\label{usefulnew2}
    Every quotient of a commutative inverse monoid by a Vagner-Preston
    congruence is a group or a group with zero adjoined.
\end{lemma}
\begin{proof}
    Let $S$ be a commutative inverse monoid and $\rho$ be a Vagner-Preston
    congruence on $S$. Clearly, $S$ is a Clifford semigroup and so is the
    quotient semigroup $S/\rho$. If for each $x\in S$, $xx^{-1}\in [1]_{\rho}$,
    then the Clifford semigroup $S/\rho$ contains the unique idempotent
    $[1]_{\rho}$, which implies that $S/\rho$ is a group. Assume that there
    exists an element $s\in S$ such that $ss^{-1}\notin [1]_{\rho}$. Fix any
    element $t\in S$ such that $tt^{-1}\notin [1]_{\rho}$. Since $\rho$ is a
    Vagner-Preston congruence and the semigroup $S$ is commutative, we get that
    $[s]_{\rho}$ and $[t]_{\rho}$ are two-sided ideals in $S$. Then
    $[s]_{\rho}\cap [t]_{\rho}\neq \emptyset$, which shows that
    $[t]_{\rho}=[s]_{\rho}$. Thus, for each $x\in S$ either $xx^{-1}\in
    [1]_{\rho}$ or $x\in [s]_{\rho}$. At this point it is easy to see that the
    quotient semigroup $S/\rho$ contains only two idempotents $[1]_{\rho}$ and
    $[s]_{\rho}$. Moreover, $[1]_{\rho}$ is an identity of $S/{\rho}$ and
    $[s]_\rho$ is zero of $S/{\rho}$. Hence $S/{\rho}$ is a group with adjoined
    zero.
\end{proof}

\subsection{Proof of \cref{I am done inventing new labels1}}

We need to show that a commutative topological inverse semigroup $S$ is
topologically isomorphic to a subsemigroup of $I_\N$ if and only if $S$ embeds
into a Tychonoff product $\prod_{n\in\N}G_n^0$, where $G_n$ is a discrete
countable group for every $n\in \N$.
\begin{proof}
($\Leftarrow$)
    Clearly, for each group $G$ the diagonal congruence on the monoid $G^0$ is a
    Vagner-Preston congruence. \cref{cool1} implies that for every countable group
    $G$ the discrete monoid $G^0$ embeds topologically into $I_\N$. \cref{prod}(ii)
    implies that for every countable family $\set{G_n}{n\in\N}$ of countable discrete
    groups the Tychonoff product $\prod_{n\in\N}G_n^0$ embeds topologically into
    $I_\N$. It follows that each topological subsemigroup $S$ of
    $\prod_{n\in\N}G_n^0$ embeds topologically into $I_\N$.

($\Rightarrow$)
    Consider a commutative inverse subsemigroup $S$ of $I_\N$ and assume that $S$
    carries the subspace topology inherited from $I_\N$. By \cref{usefulnew1},
    $S^1$ embeds topologically into $I_\N$. \cref{cool1} implies that the monoid
    $S^1$ possesses a countable family $\set{\rho_n}{n\in\N}$ of Vagner-Preston right
    congruences, each having countably many classes, such that the family
    $\set{[x]_{\rho_n},[x]_{\rho_n}^{-1}}{x\in S^1, n\in\N}$ is a subbasis of the
    topology on $S^1$.
    By the commutativity of $S^1$, each $\rho_n$ is a congruence on $S^1$. It
    follows that $[x]_{\rho_n}^{-1}=[x^{-1}]_{\rho_n}$ for every $x\in S^1$ and
    $n\in\N$. Hence the family $\set{[x]_{\rho_n}}{x\in S,n\in\N}$ is a subbasis of
    the topology on $S^1$. For each $n\in\N$ let $\mu_n=\bigcap_{i\leq n} \rho_i$.
    It is straightforward to check that $\set{[x]_{\mu_n}}{x\in S,n\in\N}$ is a basis
    of the topology on $S^1$. By \cref{usefulnew2}, for each $n\in\N$ the quotient
    semigroup $S^1/{\mu_n}$ is isomorphic either to $G_n$ or $G_n^0$ for some group
    $G_n$. By the definition of $\mu_n$ the group $G_n$ is countable. For every
    $n\in\N$ let $f_n:S^1\rightarrow G_n^0$ be the homomorphism associated with the
    congruence $\mu_n$. Since each equivalence class $[x]_{\mu_n}$ is clopen, we
    get that the homomorphism $f_n$ (into a discrete monoid $G_n^0$) is continuous.
    Similarly as in the proof of \cref{I am done inventing new labels} it can be
    checked that the diagonal map $\Phi:S^1\rightarrow \prod_{n\in\N}G_n^0$,
    $(x)\Phi=((x)f_n)_{n\in\N}$ is a topological embedding of the monoid $S^1$ into
    the Tychonoff product $\prod_{n\in\N}G_n^0$, where each factor is discrete. It
    follows that the topological semigroup $S$ is topologically isomorphic to a
    subsemigroup of $\prod_{n\in\N}G_n^0$.
\end{proof}

\subsection{Proof of \cref{theorem-compact}}
We will show that for a compact topological semigroup $S$, 
    the following conditions are equivalent:
    \begin{enumerate}[label=\rm (\roman*)]
    \item $S$ is homeomorphic to a subspace of $\N ^ \N$ (and $I_{\N})$;
    \item $S$ embeds topologically into $\N^\N$;
    \item $S$ is metrizable and totally disconnected.
\end{enumerate}
\begin{proof}
Let $S$ be a compact topological semigroup. The implications (ii) $\Rightarrow$ (i) and (i) $\Rightarrow$ (iii) are obvious.

(iii) $\Rightarrow$ (ii).
    The celebrated result of Numakura~\cite[Theorem 1]{Num} states that each
    Hausdorff compact totally disconnected topological semigroup $S$ is profinite,
    i.e. can be embedded into a Tychonoff product of finite discrete semigroups.
    Moreover, from the proof of~\cite[Theorem 1]{Num} follows that if the diagonal
    $\Delta=\set{(x,x)}{x\in S}$ is a $G_{\delta}$ subset of $S{\times}S$, then $S$
    can be embedded into the Tychonoff product of countably many finite discrete
    semigroups. Since the diagonal of any metrizable space $X$ is a $G_{\delta}$
    subset of $X{\times}X$, the result of Numakura implies that each compact
    metrizable totally disconnected topological semigroup $S$ is countably prodiscrete.
    By~\cref{cnew}, $S$ embeds topologically into $\N^\N$.
\end{proof}

\subsection{Proof of \cref{theorem-compact-2}}

We need to prove that for any compact topological inverse semigroup $S$ the
following conditions are equivalent:
\begin{enumerate}[label=\rm (\roman*)]
    \item $S$ is homeomorphic to a subspace of $\N ^ \N$ (and $I_{\N}$);
    \item $S$ embeds topologically into $I_{\N}$;
    \item $S$ embeds topologically into $\N^\N$;
    \item $S$ is metrizable and totally disconnected.
\end{enumerate}

\begin{proof}
   The equivalences (i) $\Leftrightarrow$ (iii) and (iii) $\Leftrightarrow$ (iv) are established in
    \cref{theorem-compact}.

    The implication (ii) $\Rightarrow$ (iv) follows from the fact that $I_\N$ is
    metrizable and totally disconnected.

    (iv) $\Rightarrow$ (ii).
    As we already showed in the proof of~\cref{theorem-compact}, each totally
    disconnected compact metrizable topological semigroup $S$ embeds into a
    Tychonoff product $\prod_{n\in\N}Y_{n}$ of finite discrete semigroups
    $Y_n$, $n\in\N$. For each $n\in \N$ consider the projection
    $(S)\pi_{n}\subseteq Y_n$ of $S$ on the $n$-th coordinate. Clearly, the
    semigroup $(S)\pi_n$ is inverse. By the Wagner-Preston Theorem, each
    countable inverse semigroup embeds into $I_\N$. Since for each $n\in\N$ the
    semigroup $(S)\pi_n$ is finite, the discrete semigroup $(S)\pi_n$ embeds
    topologically into $I_\N$. Taking into account that $S\subseteq
    \prod_{n\in\N}(S)\pi_n$, \cref{prod}(ii) implies that $S$ is
    topologically isomorphic to a subsemigroup of $I_\N$.
\end{proof}

Besides the canonical topology, $I_\N$ can be endowed with a Polish semigroup
topology $\Tau$ which is generated by the subbasis consisting of the sets
$U_{x, y}$ and $W_x$ (defined in the introduction), where $x,y\in \N$. The
topology $\Tau$ was investigated in~\cite{main} under the name $\mathcal I_2$, and is used in the proof of the next proposition.

\begin{proposition}\label{embcl}
    If a topological Clifford semigroup $S$ embeds topologically into $I_\N$,
    then $S$ is topologically isomorphic to a subsemigroup of $\N^\N$.
\end{proposition}

\begin{proof}
    Observe that each Clifford subsemigroup $S$ of $I_\N$ consists of partial
    permutations of subsets of $\N$, i.e.  for any $x\in S$, $\dom(x)=\im(x)$.
    Thus, for any Clifford subsemigroup $S$ of $I_\N$ the subspace topology on
    $S$ inherited from the canonical topology on $I_\N$ coincides with the
    subspace topology inherited from $\Tau$ (as defined above). Hence it remains to check that
    $(I_\N,\Tau)$ embeds topologically into $\N^\N$.  A routine verification
    shows that the map $\phi:I_{\N}\rightarrow \N ^ \N$ defined by
    \[
        (g)\phi = \set{(x + 1, y + 1)}{(x, y) \in g}\cup\set{(x + 1,0)}{x\in \N\setminus \dom(g)} \cup \{(0, 0)\}
    \]
    is a topological embedding of $(I_{\N},\Tau)$ into $\N ^ \N$.
\end{proof}

\begin{lemma}\label{ultranew}
    Let $G$ be a subgroup of $\N^\N$ and $e_G$ be the identity of $G$. Then the
    following conditions hold:
    \begin{enumerate}[label=\rm (\roman*)]
        \item $(x)e_G=x$ for every $x\in\im(e_G)$;
        \item $\im(g)=\im(f)$ for any $f,g\in G$;

        \item for any $g\in G$ the restriction $g{\restriction}_{\im(g)}$ is a
          permutation of $\im(g)$;
        \item for any $g\in G$ the restriction $g^{-1}{\restriction}_{\im(g)}$
          is equal to the inverse permutation of $g{\restriction}_{\im(g)}$ in
          $S_{\im(g)}$;
        \item for any $f,g\in G$, $f=g$ if and only if
          $f{\restriction}_{\im(f)}=g{\restriction}_{\im(g)}$;
        \item
          if $f\in G$, $x\in\N$ and $x'\in\im(f)$ be such that $(x)f=(x')f$,
          then for every $g\in G$, $(x)f=(x)g$ if and only if $(x')f=(x')g$.
    \end{enumerate}
\end{lemma}

\begin{proof}
    (i) Since $e_G$ is the identity of $G$ we obtain that $((x)e_G)e_G=(x)e_G$
    for every $x\in\N$. Then $(y)e_G=y$ for each $y\in\im(e_G)$.

    (ii) Since  $f=fg^{-1}g$ we obtain that $\im(f)\subseteq \im(g)$.
    Similarly, the equality $g=gf^{-1}f$ implies that $\im(g)\subseteq \im(f)$.
    Hence $\im(g)=\im(f)$ for each $f,g\in G$.

    (iii) Fix $g\in G$ and $x,y\in\im(g)$ such that \((x)g=(y)g\).
    By item (ii), $x,y\in \im(e_G)$.
    By item (i), $$x=(x)e_G=((x)g)g^{-1}=((y)g)g^{-1}=(y)e_G=y,$$ and so 
    the map $g{\restriction}_{\im(g)}$ is injective. Item (i) implies that
    $\im(gg)=\im(g)$. It follows that $g{\restriction}_{\im(g)}$ is a permutation.

    (iv) Follows from items (i), (ii) and (iii).

    (v) Assume that $f{\restriction}_{\im(f)}=g{\restriction}_{\im(g)}$ for
    some $f,g\in G$. Then $(gf^{-1})\restriction_{\im(g)}$ is the identity
    permutation of $\im(g)$. It is straightforward to check that $gf^{-1}$ is
    an idempotent. Since $E(\N^\N)\cap G=\{e_G\}$, we get that $gf^{-1}=e_G$,
    and so $f=g$.

    (vi) First assume that $(x')f=(x)f=(x)g$ for some $f,g\in G$, $x\in\N$ and
    $x'\in\im(f)$. Since $x'\in\im(g)$ we obtain the following:
    $$(x')g= ((x')e_G)g=((x')f)f^{-1}g= ((x)ff^{-1})g= ((x)e_G)g=((x)g)e_G=(x)g=(x')f.$$
    Assume that $(x)f=(x')f=(x')g$ for some $f,g\in G$, $x\in\N$ and $x'\in\im(f)$. Then
    \[
      (x)g = ((x)e_G)g = ((x)f) f^{-1}g = ((x')ff^{-1})g = ((x')e_G)g = (x')g =
      (x)f.\qedhere
    \]
\end{proof}

\subsection{Proof of \cref{theorem-groups}}

We need to prove that for a topological group $G$ the following conditions are
equivalent:
\begin{enumerate}[label=\rm (\roman*)]
  \item $G$ embeds topologically into $S_{\N}$;
  \item $G$ embeds topologically into $I_\N$;
  \item $G$ embeds topologically into $\N^\N$;
  \item $G$ is second-countable and has a neighbourhood basis of the identity
    consisting of open subgroups.
\end{enumerate}
\begin{proof}
    The implication (i) $\Rightarrow$ (ii) is trivial; the implication (ii)
    $\Rightarrow$ (iii) follows from \cref{embcl};
    and the equivalence (i) $\Leftrightarrow$ (iv) is well-known and follows
    from~\cite[Theorem 5.5]{main}.

    (iii) $\Rightarrow$ (i). Let \(G\) be a subgroup of \(\N^\N\). Let
    $X=\im(g)$ for some $g\in G$. \cref{ultranew}(ii) implies that the set $X$
    does not depend on the choice of $g$.
    By \cref{ultranew}(v), $G$ acts faithfully by permutations on the set $X$. This
    gives a natural algebraic embedding \(\phi: G \to S_{X}\) defined by
    $(g)\phi=g{\restriction}_{\im(g)}$. It remains to show that this embedding is
    topological.
    A subbasic open set in the topology on \((G)\phi\) has the form
    \[\set{g \in (G)\phi}{(x, y)\in g}\]
    for some \(x, y\in X\). By the definition of $\phi$, $$(\set{g \in
            (G)\phi}{(x, y)\in g})\phi ^ {-1} = \set{g \in G}{(x, y)\in g},$$
    which is open in the topology on \(G\). Hence the map $\phi$ is continuous.
    Conversely, a subbasic open set in the topology on \(G\) has the form
    \[\set{g \in G}{(x, y)\in g}\]
    for some \(x\in \N\) and \(y\in X\). By \cref{ultranew}(iii), there exists
    \(x'\in X\) such that $(x')g = y$. \cref{ultranew}(vi) implies that
    \[
      (\set{g \in G}{(x, y)\in g})\phi = \set{g \in (G)\phi}{(x', y)\in g}
    \]
    and hence the map $\phi$ is a topological embedding.
\end{proof}

\begin{definition}\label{def1}
    A topological inverse semigroup $X$ is called {\em weakly ditopological} if
    for any point $x\in X$ and a neighborhood $O$ of $x$ there are
    neighborhoods $U$, $V$ and $W$ of the points $x$, $x^{-1}x$ and $xx^{-1}$,
    respectively, such that
    $$\set{s\in S}{\exists{b}\in U,\ \exists e\in W\cap E(S) \text{ such that
    }b=es}\cap \set{s\in S}{ss^{-1}\in W}\cap\set{s\in S}{s^{-1}s\in V}\subseteq
    O.$$
\end{definition}

Clearly, each ditopological inverse semigroup is weakly ditopological. However,
the converse is not true (see \cite[Example 3.4]{P}).
Since $xx^{-1}=x^{-1}x$ in Clifford semigroups, we get the following.

\begin{proposition}\label{P}
    A Clifford semigroup $S$ is ditopological if and only if $S$ is weakly
    ditopological.
\end{proposition}

\begin{proposition}\label{wdit}
    The topological inverse semigroup $I_\N$ is weakly ditopological.
\end{proposition}

\begin{proof}
    Fix a partial injection $f\in I_\N$ and a basic open neighborhood $O$ of
    $f$. Then there exist finite subsets $A_1,A_2,A_3$ of $\N$ such that
    $A_1\subseteq \dom(f)$, $A_2\cap \dom(f)=\emptyset$, $A_3\cap
    \im(f)=\emptyset$ and
    $$O=\set{g\in I_\N}{g{\restriction}_{A_1}=f{\restriction}_{A_1},\ 
    \dom(g)\cap A_2=\emptyset \text{ and } \im(g)\cap A_3=\emptyset}.$$

    By $\id_A$ we denote the identity function on a subset
    $A\subseteq \N$. Consider the open neighborhoods $$W=\set{g\in I_\N}
        {g{\restriction}_{A_1}=\id_{A_1}\text{ and } \dom(g)\cap
        A_2=\emptyset}$$ and $$V=\set{g\in I_\N}{
        g{\restriction}_{(A_1)f}=\id_{(A_1)f}\text{ and } \im(g)\cap
        A_3=\emptyset}$$
        of $ff^{-1}$ and $f^{-1}f$, respectively. Put $U=O$ and
    \begin{multline*}
    D=\set{s\in I_{\N}}{\exists{b}\in U,\ \exists e\in W\cap E(S) \text{ such that
        }b=es} \\
        \cap \set{s\in I_\N}{ss^{-1}\in W} \cap\set{s\in I_\N}{s^{-1}s\in V}.
    \end{multline*}

    In order to show that $D\subseteq O$
    fix any $y\in D$. There exist $b\in U$ and $e\in W$ such that $b=ey$.
    Consequently, $b{\restriction}_{\dom(e)}=y{\restriction}_{\dom(e)}$. Since
    $A_1\subseteq \dom(e)$ and $f{\restriction}_{A_1}=b{\restriction}_{A_1}$ we
    deduce that $y{\restriction}_{A_1}=f{\restriction}_{A_1}$. By the definition of
    $W$, for each element $z\in\set{s\in I_\N}{ss^{-1}\in W}$ we have that
    $\dom(z)\cap A_2=\emptyset$. Consequently, $\dom(y)\cap A_2=\emptyset$.
    Analogously, for each $z\in\set{s\in I_\N}{s^{-1}s\in V}$ we have that
    $\im(z)\cap A_3=\emptyset$. It follows that $\im(y)\cap A_3=\emptyset$. Hence
    $y\in O$, and so the semigroup $I_\N$ is weakly ditopological.
\end{proof}

By $2^\N$ we denote the Cantor set endowed with the semilattice operation of
taking coordinate-wise minimum.

\begin{lemma}\label{semil}
    The semilattice of idempotents of $I_\N$ is topologically isomorphic to $2^\N$.
\end{lemma}

\begin{proof}
    Note that any idempotent $e$ of $I_\N$ is the identity map on the set
    $\dom(e)$. A routine verification shows that the map $\phi:
    E(I_\N)\rightarrow 2^{\N}$ which assigns to each element $e\in E(I_\N)$ the
    characteristic function of $\dom(e)$ is a topological isomorphism.
\end{proof}

For an element $x$ of a semilattice $X$ let
$${\Uparrow} x=\set{z\in X}{\text{there exists an open neighborhood } V \text{ of }
z \text{ such that }V\subseteq{\uparrow} x}.$$
A subset $A$ of a topological semilattice $X$ is called {\em U-dense} if for
each point $x\in X$ and a neighborhood $U$ of $x$ in $X$ there exists a point
$y\in U\cap A$ such that $x\in {\Uparrow}y$. A semilattice $X$ is called {\em
U-separable} if it possesses a countable $U$-dense subset.

The following result was proved in~\cite[Proposition 2.5]{BP}

\begin{proposition}\label{sep}
    Each second-countable $U$-semilattice is $U$-separable.
\end{proposition}

Recall that if $X$ is a topological semigroup, then $X^{0}$ carries the unique
topology $\Tau$ such that the point $0$ is isolated in $(X^0,\Tau)$ and the
subspace topology on $X$ inherited from $(X^0,\Tau)$ coincides with the
original topology on $X$. Let $X$ be an inverse semigroup and $e\in E(X)$. Then
the maximal subgroup $\set{x\in X}{xx^{-1}=e=x^{-1}x}$ is denoted by $H_e$.
The following nontrivial result proved by Banakh and Pastukhova in~\cite[Theorem 3.2]{BP}
is crucial for this paper.

\begin{theorem}\label{thBP}
    Let $S$ be a Hausdorff ditopological Clifford semigroup whose set of
    idempotents $E(S)$ is a $U_2$-semilattice, and $A$ be any U-dense subset of
    $S$. Then $S$ can be topologically embedded into the Tychonoff product
    $$E(S){\times}\prod_{e\in A} (H_e^0)^{A\cap {\Uparrow}e},$$
    where $H_e$ is endowed with the subspace topology inherited from $S$.
\end{theorem}

\subsection{Proof of \cref{theorem-we-need-better-labels}}

We need to show that a topological Clifford semigroup $S$ whose set of
idempotents $E(S)$ is a $U$-semilattice embeds topologically into $I_{\N}$ if
and only if $S$ is Hausdorff, ditopological and every maximal subgroup of $S$,
as well as the semilattice $E(S)$, embed topologically into $I_\N$.

\begin{proof}
($\Rightarrow$) If a Clifford topological semigroup $S$  embeds topologically into $I_\N$,
    then $S$ is Hausdorff. Also the maximal subgroups of $S$ and the
    semilattice $E(S)$ embed topologically into $I_\N$.   According to~\cite{P}
    each inverse subsemigroup of a weakly ditopological inverse semigroup is
    weakly ditopological. Then \cref{wdit} implies that the Clifford
    semigroup $S$ is weakly-ditopological. \cref{P} yields that $S$
    is ditopological.

($\Leftarrow$)
    Let $S$ be a Hausdorff ditopological Clifford semigroup whose set of
    idempotents $E(S)$ satisfies the following properties:
    \begin{enumerate}
        \item $E(S)$ is a $U$-semilattice which embeds topologically into $I_\N$;
        \item 
     for every $e\in E(S)$ the maximal subgroup $H_e=\set{x\in S}
    {xx^{-1}=e=x^{-1}x}$ embeds topologically into $I_\N$. 
    \end{enumerate}
    By~\cite[Proposition
        2.4(6)]{BP}, each $U$-semilattice which embeds topologically into $2^\N$ is a
    $U_2$-semilattice. \cref{semil} implies that $E(S)$ is a $U_2$-semilattice.
    Hence $S$ satisfies conditions of \cref{thBP}. Therefore, for any
    U-dense subset $A\subseteq S$, $S$ can be topologically embedded into the
    Tychonoff product $$E(S){\times}\prod_{e\in A} (H_e^0)^{A\cap {\Uparrow}e}.$$
    Since the space $I_\N$ is Polish, the semilattice $E(S)$ is second-countable.
    \cref{sep} implies that we lose no generality assuming that the set $A$ is
    countable.

    By \cref{usefulnew1}, for each idempotent $e\in S$ the topological monoid
    $H_e^0$ embeds topologically into $I_\N$. Since the set $A$ is countable,
    ~\cref{prod}(ii) implies that for every $e\in E(S)$ the topological semigroup $(H_e^0)^{A\cap {\Uparrow}e}$
    embeds topologically into $I_\N$. Using one more time~\cref{prod}(ii) we
    get that $\prod_{e\in A}(H_e^0)^{A\cap {\Uparrow}e}$ embeds topologically
    into $I_\N$.

    By the assumption, $E(S)$ is topologically isomorphic to a subsemigroup of
    $E(I_\N)$. \cref{prod}(ii) ensures that $E(S){\times}\prod_{e\in A}
    (H_e^0)^{A\cap {\Uparrow}e}$ embeds topologically into $I_\N$. Hence $S$ is
    topologically isomorphic to a subsemigroup of $I_\N$.
\end{proof}

\begin{proposition}\label{nontrivial}
    Let $S$ be a Clifford subsemigroup of $\N^\N$.  Then $S$ endowed with the
    subspace topology is a topological inverse semigroup.
\end{proposition}

\begin{proof}
    Suppose that $S$ is any Clifford subsemigroup of $\N ^ \N$. Throughout this proof it will be convenient to denote  $s{\restriction}_{\im(s)}$ by $\phi_s$ for every $s\in S$. Since every Clifford semigroup is a union of groups, if $s\in S$, then $s$ belongs to a subgroup of $\N ^ \N$. In particular, 
    by \cref{ultranew}, for each $s\in S$, $\phi_s$ is a permutation of $\im(s)$ and $\phi_s^{-1}=\phi_{s^{-1}}$, where by $\phi_s^{-1}$ we mean the inverse permutation of $\phi_s$. Although $\phi_s\not\in \N ^ \N$ unless $s$ is a permutation itself, if $t\in \N ^ \N$ is any transformation such that $\im(t)\subseteq \im(s)$, then the compositions $t\circ \phi_s$ and $t\circ \phi_s ^ {-1}$ (as binary relations) belong to $\N ^ \N$. For the sake of brevity, we will denote them as $t\phi_s$ and $t\phi_s ^ {-1}$, respectively. Taking into account that 
    $(x)\phi_s = (x)s$ for all $x\in \im(s)$, \cref{ultranew} implies the following: $s\phi_s ^ {-1}s = s$ and  $s\phi_s ^{-1} = ss^ {-1} = s ^ {-1}s = s ^ {-1}\phi_s$ for any $s\in S$.

    Clearly, it is enough to show that inversion is continuous in
    $S$. For a finite partial function
    $f$ on $\N$ consider a nonempty basic open set \(U=\set{u\in S}{f\subseteq
    u}\).
    Fix an arbitrary \(s\in U^ {-1}\). In order to show that the set \(U^{-1}\) is
    open, we need to find an open neighbourhood \(V\) of \(s\) such that
    $V\subseteq U ^ {-1}$. Let 
    $$T=(\dom(f))s\phi_s ^ {-1}
        =\set{x\in \im(s)}{(x)s\in (\dom(f))s},$$ and
    $$Z= \dom(f) \cup  T\cup  (T)\phi_s ^ {-1}.$$
    Observe that the set $Z$ is finite. 
    We define
    \[
    W' = \set{t\in S}{\forall x\in Z, (x)s\phi_s ^ {-1} \in \im(t) \text{ and }(x)t=(x)s\phi_s ^ {-1}t}    \]
    and
    \[
      W= \set{t\in S}{\im(s) \cap Z \subseteq \im(t)} 
      \cap W'.
    \]
Taking into account that for each finite subset $A\subseteq \N$ the set
    $\set{g\in\N^\N}{A\subseteq \im(g)}$ is open in $\N^\N$, it is routine to verify  that the set $W'$ is open and contains $s$. It follows that the set
    $W$ is an open neighborhood of $s$. We will show
    that
    \[
        W \subseteq  \set{t\in S}{\im(t)\cap Z=\im(s)\cap Z}.
    \]
    It suffices to check that for every $t\in W$ if $x\in Z\setminus \im(s)$,
    then $x\in Z\setminus \im(t)$. If \(x\in Z\setminus \im(s)\), then $x\in \dom(f)$ and, consequently, \(
    (x)s\phi_s ^ {-1}\in T\subseteq \im(s)\cap Z\). By the definition of $W'$, $(x)s\phi_s ^ {-1}\in \im(t)$ and $(x)t = (x)s\phi_s^{-1}t$. But $x \not\in \im(s)$ and $(x)s\phi_s ^ {-1}\in \im(s)$ and so  $x \neq (x)s\phi_s ^ {-1}$.
    On the other hand, $(x)s\phi_s^{-1}\in \im(t)$ and $t$ is a permutation of $\im(t)$, and so
    $x\not\in \im(t)$. Thus, $W
    \subseteq \set{t\in S}{\im(t)\cap Z=\im(s)\cap Z}$.%

    We define $V = W \cap \set{t\in S}{t{\restriction}_{Z} =
    s{\restriction}_{Z}}$. Clearly, $V$ is an open neighborhood of $s$. It
    remains to check that $V\subseteq U ^ {-1}$.
    Let \(t\in V\) be arbitrary. We need to show that \(f\subseteq t^{-1}\). As
    \[t\in V= W \cap \set{k\in S}{k{\restriction}_{Z} =
    s{\restriction}_{Z}}\subseteq \set{k\in S}{\im(k)\cap Z=\im(s)\cap
    Z}\cap \set{k\in S}{k{\restriction}_{Z} = s{\restriction}_{Z}},\] we
    know that \(t{\restriction}_{Z} = s{\restriction}_{Z}\) and \(\im(t)\cap
    Z=\im(s)\cap Z\).

 Let $x\in \dom(f)$ be arbitrary.  We will show that $(x)t^{-1}=(x)f$. Since $tt^{-1}=t^{-1}t$, \cref{ultranew}(iii) and (iv) imply  the following: 
    \begin{equation}\label{eq-explain-more-2}
    (x)t^{-1}= (x)t^{-1}tt^{-1} = (x)tt ^ {-1}t^{-1} = (x)t \phi_t ^ {-1}\phi_t ^ {-1}. 
    \end{equation}
    Since $s\in U ^ {-1} =\set{u\in S}{f \subseteq u }^{-1}
    = \set{u ^ {-1} \in S}{ f\subseteq u} = \set{u \in S}{f \subseteq u ^{-1}}$ we get that
    \begin{equation}\label{sec}
    (x)f=(x)s^{-1}=(x)s^{-1}ss ^{-1} = (x)ss ^{-1}s ^{-1} = (x)s\phi_s^{-1}\phi_s^{-1}.\end{equation}
 
 By assumption \(x\in Z\) and so  \((x)t=(x)s\). Since
    \((x)s\phi_s^{-1} \in T \subseteq Z\) and
    $s{\restriction}_{Z} = t{\restriction}_{Z}$, we get that
    \begin{equation}\label{eq-explain-more-1}
    (x)s\phi_s ^ {-1}t = (x)s\phi_s^{-1} s = (x)s.
    \end{equation}
    Clearly, $(x)s\phi_s^{-1} \in \im(s)\cap Z = \im(t) \cap Z$,
    and so $(x)s\phi_s^{-1}\in \im(t)$. Hence
    \begin{equation}\label{eq-explain-more-3}
    (x)s\phi_s^{-1}=(x)s\phi_s^{-1}\phi_t\phi_t^{-1}=(x)s\phi_s^{-1}t\phi_t^{-1}=(x)s\phi_t^{-1}
    \end{equation}
    (the last equality holds by \eqref{eq-explain-more-1}). 
    Similarly,
    $(x)s\phi_s^{-1}\phi_s^{-1}\in (T)\phi_s^{-1}\subseteq Z\cap
    \im(s)$  implies $(x)s\phi_s^{-1}\phi_s^{-1}\in \im(t)$. Since $s{\restriction}_Z = t{\restriction}_Z$ we have that
 $(x)s\phi_s^{-1}\phi_s^{-1}t
    =(x)s\phi_s^{-1}\phi_s^{-1}s
    =(x)s\phi_s^{-1}.$
    So   
    \begin{equation}\label{eq-explain-more-4}(x)s\phi_s^{-1}\phi_s^{-1}  
    = (x)s\phi_s^{-1}\phi_s^{-1}t \phi_t ^ {-1}
    = (x)s\phi_s^{-1}\phi_t ^ {-1}.
    \end{equation}
    Hence
    \begin{equation*}
    \begin{array}{rcll}
        (x)t^{-1} & = & (x)t\phi_t^{-1}\phi_t^{-1} &  \text{by }\eqref{eq-explain-more-2}\\ 
        &= &(x)s\phi_t^{-1}\phi_t^{-1} & \hbox{as } x\in\dom(f)\subseteq Z \hbox{ and } s{\restriction}_Z = t{\restriction}_Z\\
        &= & (x)s\phi_s^{-1}\phi_t^{-1} & \text{by }\eqref{eq-explain-more-3}\\
        & =&  (x)s\phi_s^{-1}\phi_s^{-1} & \text{by }\eqref{eq-explain-more-4} \\
        & = & (x)f & \text{by }\eqref{sec}.\\
         \end{array}
         \end{equation*}
Since $x\in \dom(f)$ was arbitrary, we obtain that $f\subseteq t^{-1}$, and so $t\in U ^ {-1}$. It follows that $V\subseteq U^{-1}$. Thus, the set $U^{-1}$ is open, implying the continuity of inversion in $S$.
\end{proof}

By the proof of \cref{embcl}, the inverse topological semigroup $(I_\N,\Tau)$
embeds topologically into $\N^\N$. However the inversion is not continuous in
$(I_\N,\Tau)$. Hence \cref{nontrivial} doesn't hold for an arbitrary inverse
subsemigroup of $\N^\N$.

\begin{proposition}\label{new 2}
    Each topological inverse subsemigroup of $\N^\N$ is ditopological.
\end{proposition}

\begin{proof}
    Let $S$ be a topological inverse subsemigroup of $\N^\N$. By $\Tau$ we
    denote the subspace topology on $S$ inherited from $\N^\N$. By~\cref{cool},
    there exists a family $\set{\rho_n}{n\in\N}$ of right congruences on $S$, each
    having countably many equivalence classes, such that the family
    $\set{[x]_{\rho_n}}{x\in S,\ n\in\N}$ is a basis of $\Tau$.

    To show that $S$ is ditopological fix any $x\in S$ and an open neighborhood $O$
    of $x$. Then there exists $n\in\N$ such that $[x]_{\rho_{n}}\subseteq O$. Let
    $U=[x]_{\rho_n}$ and $W=[xx^{-1}]_{\rho_n}$. It remains to check that
    $$D=\set{y\in S}{\exists{b}\in U,\ \exists e\in W\cap E(S) \text{ such that
        }b=ey}\cap \set{y\in S}{yy^{-1}\in W}\subseteq O.$$ Fix any $y\in D$. Then there
    exist $b\in U$ and $e\in W\cap E(S)$ such that $b=ey$. The choice of the sets
    $U$ and $W$ implies that $(b,x)\in\rho_n$ and $(e,xx^{-1})\in\rho_n$. Since
    $yy^{-1}\in W$ we get that $(xx^{-1},yy^{-1})\in\rho_n$ and, consequently,
    $(e,yy^{-1})\in\rho_n$. Then
    $$[x]_{\rho_n}=[b]_{\rho_n}=[ey]_{\rho_n}=[yy^{-1}y]_{\rho_n}=[y]_{\rho_n}.$$
    Hence $y\in U$, and so $S$ is a ditopological inverse semigroup.
\end{proof}

\subsection{ Proof of \cref{ditop}}

We need to show that each Clifford subsemigroup of $\N^\N$ is ditopological.

\begin{proof}
 By \cref{nontrivial}, each Clifford subsemigroup of $\N^\N$ is a topological inverse semigroup. \cref{new 2} implies that Clifford subsemigroups of $\N^\N$ are ditopological.   
\end{proof}

\subsection{Proof of \cref{newtheorem1}}

We need to show that a Clifford topological semigroup $S$ whose set of
idempotents $E(S)$ is a $U_2$-semilattice embeds topologically into $\N^{\N}$
if and only if $S$ is Hausdorff, ditopological and every maximal subgroup of
$S$, as well as the semilattice $E(S)$, embed topologically into $\N^\N$.

\begin{proof}
($\Rightarrow$)    Let $S$ be a Clifford topological subsemigroup of $\N^\N$.  Then $S$ is
    Hausdorff and every maximal subgroup of $S$, as well as the semilattice
    $E(S)$,  embeds topologically into $\N^\N$. \cref{ditop}
 implies that $S$ is ditopological.

($\Leftarrow$) Let $S$ be a Hausdorff ditopological Clifford semigroup whose set of
    idempotents $E(S)$ is a $U_2$-semilattice which embeds topologically into
    $\N^\N$ and for every $e\in E(S)$ the maximal subgroup $H_e=\set{x\in S}
    {xx^{-1}=e=x^{-1}x}$ embeds topologically into $\N^\N$. Then $S$ satisfies
    conditions of \cref{thBP}. Therefore, for any U-dense subset $A\subseteq S$,
    $S$ can be topologically embedded into the Tychonoff product
    $$E(S){\times}\prod_{e\in A} (H_e^0)^{A\cap {\Uparrow}e}.$$
    Since the space $\N^\N$ is Polish, the semilattice $E(S)$ is second-countable.
    By \cref{sep} we can assume that the set $A$ is countable.

    By \cref{usefulnew}, for each idempotent $e\in S$ the topological monoid
    $H_e^0$ embeds topologically into $\N^\N$. Since the set $A$ is countable,
    ~\cref{prod}(i) implies that for every $e\in E(S)$ the topological semigroup $(H_e^0)^{A\cap {\Uparrow}e}$ embeds
    topologically into $\N^\N$. Using one more time~\cref{prod}(i) we get that
    $\prod_{e\in A}(H_e^0)^{A\cap {\Uparrow}e}$ embeds topologically into $\N^\N$.
    By the assumption, $E(S)$ is topologically isomorphic to a subsemigroup of
    $\N^\N$. \cref{prod}(i) ensures that $E(S){\times}\prod_{e\in A} (H_e^0)^{A\cap
        {\Uparrow}e}$ embeds topologically into $\N^\N$. Hence $S$ is topologically
    isomorphic to a subsemigroup of $\N^\N$.
\end{proof}

A space $X$ is called {\em scattered} if every subset $A$ of $X$ contains an
isolated (in the subspace topology) point.
Recall that Cantor-Bendixson derivatives of a scattered space $X$ are defined by
transfinite induction as follows, where $X'$ is the set of all accumulation
points of $X$:
\begin{enumerate}[label=\rm (\roman*)]
    \item $X^{0}=X$;
    \item $X^{\alpha+1}=\big(X^{\alpha}\big)'$;
    \item $X^{\alpha}=\bigcap_{\beta<\alpha}X^{\beta}$, if $\alpha$ is a limit
      ordinal.
\end{enumerate}
The set $X^{\alpha}\setminus X^{\alpha+1}$ is called the $\alpha$-th {\em
Cantor-Bendixson level} of $X$ and is denoted by $X^{(\alpha)}$ (this notation is used in the proof of \cref{supernew}). The {\em
height} of a scattered space $X$ is the smallest ordinal $ht(X)$ such that
$X^{ht(X)}=\emptyset$.

A semilattice $X$ endowed with a topology $\Tau$ is called:
\begin{enumerate}[label=\rm (\roman*)]
    \item {\em semitopological} if for every $a\in X$ the shift $l_a:
      X\rightarrow X$, $(x)l_a=xa$ is continuous;

    \item {\em $U_2$-semilattice at a point $x$} if for every open
      neighborhood $V$ of $x$ there exist a point $y\in V$ and a clopen ideal
      $I\subseteq E$ such that $x\in X\setminus I\subseteq {\uparrow}y$.
\end{enumerate}

\begin{lemma}\label{supernew}
    Each scattered $T_1$ semitopological semilattice is a $U_2$-semilattice.
\end{lemma}

\begin{proof} 
It is easy to check that
  for each open subset $U$ of $X$ the upper set ${\uparrow} U=\bigcup_{x\in
  U}{\uparrow} x$ is open. Since each
  singleton is closed in $X$ and the semilattice $X$ is semitopological, the
  set ${\uparrow}x$ is  closed for every $x\in X$. Hence for every $x\in X^{(0)}$ 
  (the $0$-th Cantor-Bendixson level of $X$)
  the set ${\uparrow}x$ is clopen. Then for each $x\in X^{(0)}$ the
  clopen ideal $I=X\setminus {\uparrow} x$ together with the point $x$ implying 
  that $X$ has the $U_2$ property at $x\in X^{(0)}$. Assume that for some
  ordinal $\alpha<ht(X)$, $X$ has the $U_2$ property at every $x\in
  \bigcup_{\xi\in\alpha}X^{(\xi)}$. Fix any $x\in X^{(\alpha)}$ and open
  neighborhood $U$ of $x$. Since the space $X$ is scattered, we lose no
  generality assuming that $U\subseteq \bigcup_{\xi\leq \alpha}X^{(\xi)}$ and
  $U\cap X^{\alpha}=\{x\}$. The continuity of shifts in $X$ yields the
  existence of an open neighborhood $V$ of $x$ such that $xV\subseteq U$. There
  are two cases to consider:
  \begin{enumerate}
    \item $xz=x$ for all $z\in V$;
    \item there exists $z\in V$ such that $xz=y\in U\setminus \{x\}$.
  \end{enumerate}
  In case 1 we have that $V\subseteq {\uparrow}x$. Clearly, the set ${\uparrow}x$ is closed. Let us show that the upper
  cone ${\uparrow}x$ is open. Pick any $a\in{\uparrow}x$. The continuity of
  shifts in $X$ yields an open neighborhood $W$ of $a$ such that $Wx\subseteq
  V\subseteq {\uparrow}x$, establishing that the element $a$ belongs to the
  interior of ${\uparrow}x$. Hence the set ${\uparrow}x$ is clopen.  Thus, the clopen ideal $I=X\setminus
  {\uparrow} x$ together with the point $x$  show that $X$ is a
  $U_2$-semilattice at $x$.

  Assume that case 2 holds. The choice of $U$ implies that
  $y\in\bigcup_{\xi\in\alpha}X^{(\xi)}$. By the inductive assumption, there
  exist $p\in U$ and a clopen ideal $I$ such that $y\in X\setminus I\subseteq
  {\uparrow} p$. Note that $$xp=x(yp)=(xy)p=xxzp=xzp=yp=p,$$
  implying that $x\in{\uparrow}p$. Since $xy=xxz=xz=y$ we get that $x\notin I$, because otherwise $y=xy\in I$, contradicting the choice of $I$.
  Thus, the point $p\in U$ and the clopen ideal $I$ prove that the semilattice
  $X$ has the $U_2$ property at the point $x$.

  Hence $X$ is a $U_2$-semilattice at each point $x\in X$, implying that $X$ is
  a $U_2$-semilattice.
\end{proof}

\subsection{Proof of \cref{emb22}}

We need to show that a countable Polish Clifford semigroup $S$ embeds
topologically into $I_\N$ if and only if $S$ is ditopological and the
semilattice $E(S)$ embeds topologically into $I_\N$.

\begin{proof}
($\Rightarrow$)
    According to~\cite{P} each inverse subsemigroup of a weakly ditopological
    inverse semigroup is weakly ditopological. Then \cref{wdit}
    implies that each Clifford subsemigroup $S$ of $I_\N$ is
    weakly ditopological. \cref{P} yields that $S$ is ditopological.
    Clearly, since the entire semigroup $S$ embeds into $I_\N$, so too does its semilattice of idempotents.

($\Leftarrow$)
    Let $S$ be a ditopological countable Polish Clifford semigroup such that the
    semilattice $E(S)$ embeds topologically into $I_\N$. The continuity of the
    inversion in $S$ implies that maximal subgroups of $S$ are closed and hence
    Polish. Since non-discrete Polish topological groups are of cardinality
    continuum, we deduce that each maximal subgroup of $S$ is discrete and, thus,
    embeds topologically into $I_\N$ by \cref{theorem-groups}. Clearly, every
    countable Polish space is scattered. By \cref{supernew}, the semilattice $E(S)$
    is a $U_2$-semilattice and, consequently, a $U$-semilattice.
    \cref{theorem-we-need-better-labels} implies that $X$ embeds
    topologically into $I_\N$.
\end{proof}

\subsection{Proof of \cref{emb23}}

We need to show that a countable Polish Clifford semigroup $S$ embeds
topologically into $\N^\N$ if and only if $S$ is ditopological and the
semilattice $E(S)$ embeds topologically into $\N^\N$.

\begin{proof}
  ($\Rightarrow$)
  By \cref{ditop} each Clifford subsemigroup $S$
  of $\N^\N$ is ditopological. Clearly the semilattice of idempotents of $S$ embeds in $\N ^ \N$, since $S$ embeds into $\N ^ \N$.
  
($\Leftarrow$)
  Let $S$ be a countable Polish ditopological Clifford semigroup such that the
  semilattice $E(S)$ embeds topologically into $\N^\N$. Similarly as in the proof
  of \cref{emb22} it can be checked that each maximal subgroup of $S$ embeds
  topologically into $\N^\N$ and the semilattice $E(S)$ is a $U_2$-semilattice.
  \cref{newtheorem1} implies that $X$ embeds topologically into $\N^\N$.
\end{proof}

\section{Counterexamples}\label{section-counter-examples}
In this section we collect counterexamples to \cref{question-main} as well as
other examples which show the sharpness of the results proved in the previous
section.

Given \cref{theorem-compact-2} and \cref{embcl}, it might be tempting to think
that $I_{\N}$ embeds topologically in $\N ^ \N$. The following lemma shows that
this is not the case.

\begin{proposition}
\label{Luke}
    The topological inverse semigroup $I_\N$ cannot be topologically embedded into $\N^\N$.
\end{proposition}

\begin{proof}
    Seeking a contradiction, we suppose that there is such an embedding. It follows
    from \cref{cool}, that there is a countable family \(\{\rho_i: i\in \N\}\)
    of right congruences on \(I_{\N}\) such that the equivalence classes of
    these congruences form a basis for the canonical topology on $I_\N$.

    Clearly, the set \(U = \set{f\in I_\N}{0 \notin
        \im(f)}\) is an open neighborhood of $\emptyset$ in $I_\N$.
    Then there is \(k\in \N\) such that \([\emptyset]_{\rho_k}\subseteq U\). Since the
    set \([\emptyset]_{\rho_k}\) is open in $I_\N$, there are finite subsets \(X,
    Y\subseteq \N\) such that $$V= \set{g\in I_\N}{\dom(g)\cap X = \emptyset,
        \im(g)\cap Y = \emptyset}$$ satisfies
    \[\emptyset \in V \subseteq [\emptyset]_{\rho_k}\subseteq U.\]
    Fix an arbitrary \(n\in \N\setminus(X\cup Y)\). We have that \(\{(n, n)\}\in V
    \subseteq [\emptyset]_{\rho_k}\).
    Since $\rho_k$ is a right congruence we get that
    \begin{align*}
        \{(n, 0)\} = \{(n, n)\}\circ \{(n, 0)\}
        \in [\emptyset\circ \{(n, 0)\}]_{\rho_k}
        =  [\emptyset]_{\rho_k}
        \subseteq U = \set{f\in I_\N}{0 \notin \im(f)},
    \end{align*}
    which is a contradiction.
\end{proof}

Each semilattice $X$ carries a natural partial order $\leq$ defined by $e\leq
f$ if $ef=e$ for any $e,f\in X$. 
A semilattice $X$ is called {\em chain-finite} if every
linearly ordered subset in $(X,\leq)$ is finite. The following proposition shows that \cref{embcl}
cannot be reversed.

\begin{proposition}\label{chain-finite}
    Each countable infinite discrete chain-finite semilattice $X$ embeds
    topologically into $\N^\N$, but not into $I_\N$.
\end{proposition}

\begin{proof}
Since the semilattice $X$ is countable and discrete, \cref{disc} implies that $X$ embeds topologically into $\N^\N$.
Lemma~4.1 from \cite{BB1} implies that for each topological embedding $\phi$ of $X$ into a zero-dimensional Hausdorff topological semigroup $Y$ the image $(X)\phi$ is closed in $Y$.
Consequently, for each topological embedding $\phi: X\rightarrow I_\N$ the image $(X)\phi\subseteq E(I_\N)$ is closed.
 Since the semilattice $X$ is  noncompact, \cref{semil} yields that $X$ cannot be embedded topologically into $I_\N$.
\end{proof}

Looking at Theorems~\ref{theorem-we-need-better-labels} and~\ref{newtheorem1}
it is natural to ask whether each subsemilattice of $I_\N$ or $\N^\N$ is a
$U$-semilattice. The following lemma gives a negative answer to this question.

\begin{proposition}\label{ulemma}
    There exists a topological semilattice $X$ which is not a $U$-semilattice
    at any of its points, but embeds topologically into $I_\N$ and $\N^\N$.
\end{proposition}

\begin{proof}
    Let $\mathbb R$ be the real line endowed with the usual topology and the
    semilattice operation of taking the minimum. Consider the discrete subsemilattice
    $Z=\set{\frac{1}{n+1}}{n\in \N}\cup \set{2-\frac{1}{n+1}}{n\in\N}$ of $\mathbb R$.
    Note that $\overline{Z}=Z\cup\{0,2\}$ is a compact zero-dimensional topological
    semilattice. Let $$X=\set{(x_n)_{n\in\N}\in Z^{\mathbb N}}{x_n=1 \text{ for all
            but finitely many } n\in\N }$$ 
    be a subsemilattice of the Tychonoff product
    $Z^{\N}$. It is easy to see that $Y=\overline{Z}^{\N}$ is a compact metrizable
    zero-dimensional topological semilattice which contains $X$. By
    \cref{theorem-compact-2}, $Y$ embeds topologically into $I_\N$. It follows that
    $X$ embeds topologically into $I_\N$ and hence into $\N^\N$ as well, by
    \cref{embcl}.

    Fix a point $(x_n)_{n\in\N}\in X$, an open neighborhood $U$ of $x$ and a point
    $(y_n)_{n\in\N}\in U$. In order to show that $X$ is not a $U$-semilattice at
    $x$ consider a basic open neighborhood $V=\{(z_n)_{n\in\N}\in X: z_n=x_n$ for all $n\leq k\}$ of $x$. Let $t\in Z$ be such that $t<y_{k+1}$. It is clear that
    $(x_0,\ldots, x_n,t,1,\ldots, 1,\ldots)\in V\setminus {\uparrow}y$ implying
    that $x$ doesn't belong to the interior of ${\uparrow}y$. Since $y$ was chosen arbitrarily, $X$ is not a $U$-semilattice at the point $x$.
\end{proof}

Let $X$ be a non-empty topological space. The {\em strong Choquet game} on $X$
is defined as follows: Player I chooses a pair $(x_0,U_0)$ where $U_0$ is an
open subset of $X$ and $x_0\in U_0$. Player II responds with an open subset
$V_0$ such that $x_0\in V_0\subseteq U_0$. At stage $n$ Player I chooses a pair
$(x_n,U_n)$ such that $U_n$ is an open subset of $X$ and $x_n\in U_n\subseteq
    V_{n-1}$. Player II responds with an open set $V_n\subseteq U_n$ which contains
$x_n$. If $\bigcap_{n\in\omega}U_n=\emptyset$, then Player I wins. Otherwise,
Player II wins. The following result was proved in~\cite{Ch}.

\begin{theorem}\label{choquet}
    A metrizable space $X$ is completely metrizable if and only if Player II
    has a winning strategy in strong  Choquet game on $X$.
\end{theorem}

The following lemma is helpful in detecting scattered Polish spaces.

\begin{lemma}\label{Pol}
    A scattered space $X$ is Polish if and only if $X$ is regular and
    second-countable.
\end{lemma}

\begin{proof}
($\Rightarrow$)
    Clearly, each Polish space is regular and second-countable.

($\Leftarrow$)
    By the Urysohn Metrization Theorem, each regular second-countable space $X$ is
    metrizable and separable. By \cref{choquet}, to prove that $X$ is Polish it suffices
    to show that Player II has a winning strategy in the strong Choquet game on
    $X$. Assume that we are at stage $n$ of the strong Choquet game and Player I
    chose a corresponding pair $(U_n,x_n)$. Find an ordinal $\alpha\in ht(X)$ such
    that $x_n$ belongs to the Cantor-Bendixson level $X^{(\alpha)}$. Player II can
    respond with any open neighborhood $V_n\subseteq U_n$ of $x_n$ such that $V_n
        \cap X^{(\alpha)} = V_n\cap X^{\alpha}=\{x_n\}$. Then, using the fact that
    ordinals do not possess infinite decreasing sequences, it is straightforward to
    check that this is a winning strategy for Player II.
\end{proof}

In the following four propositions we construct counterexamples to
\cref{question-main}.

\begin{proposition}\label{exB}
    There exists a countable commutative (and hence Clifford) Hausdorff
    topological inverse semigroup $S$ such that
    \begin{enumerate}[label=\rm (\roman*)]
        \item $S$ is locally compact and Polish;
        \item the semilattice $E(S)$ is compact;
        \item $S$ cannot be topologically embedded into $\N^{\N}$.
    \end{enumerate}
\end{proposition}

\begin{proof}
    By $T$ we denote the set $\{0\} \cup \set{x_n}{n\in \N}$ endowed with the
    semilattice operation
    $$ab=
        \begin{cases}
            a, \quad\hbox{if } a=b; \\
            0, \quad\hbox{otherwise}.
        \end{cases}
    $$
    Let $\{1,-1\}$ be the two-element multiplicative group. Then the direct
    product $S=T{\times}\{1,-1\}$ is a countable commutative inverse semigroup
    whose semilattice $E(S)$ coincides with the set $T {\times}\{1\}$. We endow
    $S$ with the topology $\Tau$ defined as follows:
    \begin{enumerate}
        \item each element $x\in S\setminus \{(0,1)\}$ is isolated;
        \item the sets $U_n=\{(0,1)\}\cup\set{(x_i,1)}{i\geq n}$, $n\in\N$ form an open neighborhood basis at $(0,1)$.
    \end{enumerate}
    It is easy to check that $(S,\Tau)$ is a locally compact regular scattered
    second-countable topological inverse semigroup and the semilattice
    $E(S)=T{\times}\{1\}$ is compact. By~\cref{Pol}, the space $(S,\Tau)$ is
    Polish.  To derive a contradiction, assume that $(S,\Tau)$ can be embedded
    topologically into $\N^\N$. By \cref{cool}, there exists a family
    $\set{\rho_n}{n\in\N}$ of right congruences on $S$ such that the collection
    $\set{[x]_{\rho_n}}{x\in S, n\in\N}$ is a base of the topology $\Tau$. Since
    $S$ is commutative, $\rho_n$ is a two-sided congruence for every $n\in \N$.
    It follows that for any $n,m\in\N$, $((x_n,1), (0,1))\in \rho_m$ if and
    only if $((x_n,-1),(0,-1))\in \rho_m$. By the definition of $\Tau$, for
    each $m\in \N$ there exists $q_m\in\N$ such that $((x_n,1),(0,1))\in
    \rho_m$ for every $n\geq q_m$. It follows that for each $m\in \N$,
    $((x_n,-1),(0,-1))\in \rho_m$ for every $n\geq q_m$, and so the
    point $(0,-1)$ is not isolated in  $(S,\Tau)$. The obtained contradiction
    implies that $(S,\Tau)$ cannot be topologically embedded into $\N^{\N}$.
\end{proof}

\begin{proposition}\label{locally_compact_ditosemiex}
    There exists a countable linearly ordered locally compact Polish
    topological semilattice $X$ which cannot be topologically embedded into
    $\N^{\N}$.
\end{proposition}

\begin{proof}
    Let $X$ be the semilattice $(\set{\frac{1}{n+1}}{n\in\N}\cup\{0\},\min)$
    endowed with the topology $\Tau$ which is defined as follows: each non-zero
    element of $X$ is isolated and an open neighborhood basis at $0$ consists
    of the sets $U_m=\set{\frac{1}{2n+1}}{n\geq m}\cup\{0\}$, $m\in \N$. One can
    easily check that $X$ is a locally compact regular scattered
    second-countable linearly ordered topological semilattice. By \cref{Pol},
    the space $(X,\Tau)$ is Polish. To derive a contradiction, assume
    that $X$ is  topologically isomorphic to a subsemigroup of $\N^{\N}$.
    Taking into account the commutativity of $X$, \cref{cool} yields the
    existence of a family $\set{\rho_n}{n\in\N}$ of congruences on $X$ such that
    the collection $\set{[x]_{\rho_n}}{x\in X, n\in\N}$ is a basis of the topology
    $\Tau$. Then there exists $n\in \N$ such that
    $[0]_{\rho_n}\subseteq U_1=\set{\frac{1}{2n+1}}{n\geq
    1}\cup\{0\}$. Observe that if there exists $m\in \N$   such
    that $\frac{1}{2m+1}\in [0]_{\rho_n}$, then $\frac{1}{2m+2}\in
    [0]_{\rho_n}\setminus U_1$, which contradicts our assumption. Otherwise,
    $[0]_{\rho_n}=\{0\}$, and so $0$ is an isolated point in
    $(X,\Tau)$, which contradicts the definition of $\Tau$. The obtained
    contradictions imply that $X$ is not topologically isomorphic to a
    subsemigroup of $\N^{\N}$.
\end{proof}

\begin{proposition}\label{rs}
    Let $S$ be a right simple semigroup. Then there exists no topology $\Tau$
    on $S^0$ such that $0$ is not isolated and  $(S^0,\Tau)$  embeds
    topologically into $\N^\N$.
\end{proposition}

\begin{proof}
    Let $S$ be a right simple semigroup. Fix a right congruence $\rho$ on $S^0$
    such that $[0]_{\rho}$ is not singleton. Then there exists $a\in X$ such
    that $(0,a)\in\rho$. By the right simplicity of $S$, $aS=S$. It follows
    that for every $b\in S$ there exists $c\in S$ such that $b=ac$.
    Then $[b]_{\rho}=[ac]_{\rho}=[0c]_{\rho}=[0]_{\rho}$. Hence for each right
    congruence $\rho$ on $S^0$ the equivalence class $[0]_{\rho}$ is either
    singleton or coincides with $S^0$. By \cref{cool}, if $0$ is not isolated
    in $(S^0,\Tau)$, then  $(S^0,\Tau)$ doesn't embed topologically into
    $\N^\N$.
\end{proof}

 A semigroup $S$ is called {\em congruence-free} if $S$ admits only trivial (diagonal and universal) two-sided congruences.

\begin{proposition}\label{polish_inv_compact_semi_ex}
    There exists a countable congruence-free Hausdorff locally compact Polish
    topological inverse semigroup $S$ with a compact semilattice of idempotents
    which cannot be topologically embedded into $\N^{\N}$.
\end{proposition}
\begin{proof}
    Let $S$ be the subsemigroup of $I_\N$ which consists of all partial
    bijections of cardinality $\leq 1$, and $\Tau$ be the topology on $S$ which
    satisfies the following conditions:
    \begin{enumerate}
        \item each nonempty partial bijection is isolated in $(S,\Tau)$;
        \item the sets $U_{k}=\{\emptyset\}\cup\set{\{(n,n)\}}{n\geq k}$, $k\in\N$ form an open
              neighborhood basis at $\emptyset$.
    \end{enumerate}
    Clearly the semigroup $S$ is isomorphic to the countable Brandt semigroup
    over the trivial group. By~\cite[Theorem 2]{Preston}, $S$ is congruence
    free. One can check that $(S,\Tau)$ is a regular locally compact
    second-countable topological inverse semigroup, and the semilattice
    $E(S)=\set{(n,n)}{n\in \omega}\cup\{\emptyset\}$ is compact.  \cref{Pol} implies
    that the space $(S,\Tau)$ is Polish.

    Assume that $(\{(n,n)\},\emptyset)\in \rho$ for some right congruence $\rho$ on $S$.
    Then for any $m\in\N$ we have that
    $$(\{(n,m)\},\emptyset)=(\{(n,n)\}\circ\{(n,m)\},\emptyset \circ
        \{(n,m)\})\in\rho.$$ Consequently, $\set{\{(n,m)\}}{m\in \N}\subseteq [0]_{\rho}$.
    Hence for each right congruence $\rho$ on $S$ the inclusion
    $[\emptyset]_{\rho}\subseteq \{\{(n,n)\}:n\in\N\}\cup\{\emptyset\}$ implies
    that $[\emptyset]_{\rho}=\{\emptyset\}$. \cref{cool} yields that the
    topological inverse semigroup $(S,\Tau)$ cannot be topologically embedded
    into $\N^{\N}$.
\end{proof}


\begin{thebibliography}{}

\bibitem{B}
    T. Banakh, {\em On cardinal invariants and metrizability of topological
    inverse Clifford semigroups}, Topology Appl. {\bf 128}:1 (2003), 13--48.

    \bibitem{BB1}
    T. Banakh, S. Bardyla, {\em Characterizing categorically closed commutative
    semigroups}, J. Algebra {\bf 591} (2022), 84--110.


    \bibitem{Bok}
    T. Banakh, B. Bokalo, {\em On cardinal invariants and metrizability of
    topological inverse semigroups}, Topology Appl. {\bf 128}: (2003), 3--12.

    \bibitem{BG}
    T. Banakh, O. Gutik, {\em On the continuity of inversion in countably
    compact inverse topological semigroups}, Semigroup Forum {\bf 68} (2004),
    411--418.

    \bibitem{BGPR}
    T. Banakh, O. Gutik, O. Potiatynyk and A. Ravsky, {\em Metrizability of
    Clifford topological semigroups}, Semigroup Forum {\bf 84}:2 (2012),
    301--307.

    \bibitem{BP}
    T.~Banakh, I. Pastukhova, {\em  On topological Clifford semigroups
    embeddable into products of cones over topological groups}, Semigroup Forum
    {\bf 89} (2014), 367--382.

    \bibitem{BR} T.~Banakh, A.~Ravsky, {\em  On feebly compact paratopological groups}, Topology Appl. {\bf 284} (2020), 107363.

    \bibitem{Bor}
 M. Bodirsky, F. M. Schneider, {\em A topological characterisation of endomorphism monoids of countable
structures}, Algebra Universalis, {\bf 77}:3 (2017), 251--269.

    \bibitem{Brand}
    N. Brand, {\em Another note on the continuity of the inverse}, Archiv Math.
    {\bf 39} (1982), 241--245.

    \bibitem{CHK}
    J. H. Carruth, J. A. Hildebrant, R. J. Koch, {\em The theory of topological semigroups}, volume 2, Marcel
    Dekker, 1986.

    \bibitem{Ch} G. Choquet, {\em Lectures on Analysis I.}, Benjamin, New York, 1969.

    \bibitem{Vmill}
    J. Dijkstra, J. van Mill, G. Stepr\~{a}ns, {\em Complete Erd\H{o}s space is
    unstable}, Math. Proc. Cambridge Philos. Soc. {\bf 137} (2004), 465--473.

\bibitem{EMP}
J. East, J. D. Mitchell, Y. P\'eresse, {\em Maximal subsemigroups of the semigroup of all mappings on an infinite set}, Trans. Amer. Math. Soc. {\bf 367} (2015), 1911--1944. 

\bibitem{main}
L. Elliott, J. Jonu\v{s}as, Z. Mesyan, J. D. Mitchell, M. Morayne and Y. P\'eresse, {\em Automatic continuity, unique Polish topologies, and Zariski topologies on monoids and clones}, Trans. Amer. Math. Soc. (accepted) https://doi.org/10.1090/tran/8987.  

\bibitem{EJMPP}
L. Elliott, J. Jonu\v{s}as, J. D. Mitchell, Y. P\'eresse and M. Pinsker, {\em Polish topologies on endomorphism monoids of relational structures}, Adv. Math. (accepted),  arXiv:2203.11577.

    \bibitem{Ellis}
    R. Ellis, {\em A note on the continuity of the inverse}, Proc. Amer. Math.
    Soc. {\bf 8} (1957), 372--373.

    \bibitem{Gan}
    O. Ganyushkin, V. Mazorchuk, {\em Combinatorics of nilpotents in symmetric
    inverse semigroups}, Annals of Combinatorics {\bf 8}:2 (2004), 161--175.

\bibitem{Gaughan}
E. D. Gaughan, {\em Topological Group Structures of Infinite Symmetric Groups},
Proceedings of the National Academy of Sciences of the United States of America,
{\bf 58}:3 (1967) 907--910.

\bibitem{GR}
    O. Gutik, D. Repov\v{s}, {\em The continuity of the inversion and the
    structure of maximal subgroups in countably compact topological semigroups},
    Acta Math. Hung. {\bf 124}:3 (2009), 201--214.

\bibitem{Howie1995aa}
J. M. Howie, {\em Fundamentals of semigroup theory}, 
Clarendon Press Oxford, 
London Mathematical Society Monographs, {\bf 12}, 1995.

\bibitem{Kechris1995}
  A. S. Kechris,
  Classical Descriptive Set Theory, volume 156 of Graduate Texts in Mathematics. Springer-Verlag, New York, 1995.

    \bibitem{Kud}
    G. Kudryavtseva, V. Mazorchuk, {\em Schur-Weyl dualities for symmetric
    inverse semigroups}, J. Pure Appl. Algebra {\bf 212}:8 (2008), 1987--1995.

    \bibitem{Law}
    M. Lawson, {\em Inverse Semigroups. The Theory of Partial Symmetries},
    World Scientific, Singapore, 1998.

    \bibitem{Dragan}
    D. Mas\v{s}ulovi\'c, M. Pech, {\em Oligomorphic transformation monoids and
    homomorphism-homogeneous structures}, Fundamenta Mathematicae {\bf 212}:1
    (2011), 17--34.

\bibitem{Mesyan1}
Z. Mesyan, {\em Monoids of injective maps closed under conjugation by permutations}, Israel J. Math.
{\bf 189}:1 (2011), 287--305.

\bibitem{Mesyan2}
Z. Mesyan. Generating self-map monoids of infinite sets. Semigroup Forum {\bf 75}:3 (2007), 648--675.

\bibitem{MMMP}
Z. Mesyan, J. Mitchell, M. Morayne, and Y. P\'eresse, {\em The Bergman-Shelah preorder on transformation
semigroups}, Mathematical Logic Quarterly {\bf 58} (2012), 424--433.

\bibitem{MS2023}
M. Megrelishvili and M. Shlossberg, {\em Non-archimedean topological monoids},
2023, arXiv:2311.09187,

    \bibitem{Mon}
    D. Montgomery, {\em Continuity in topological groups}, Bull. Amer. Math.
    Soc. {\bf 42} (1936), 879--882.

    \bibitem{Num}
    K. Numakura, {\em Theorems on compact totally disconnected semigroups and
    lattices}, Proc. Amer. Math. Soc. {\bf 8}:4 (1957) 623--626.

    \bibitem{P}
    I. Pastukhova, {\em Automatic continuity of homomorphisms between
    topological inverse semigroups}, Topol. Algebra Appl. {\bf 6} (2018)
    60--66.

    \bibitem{PP}
    C. Pech, M. Pech, {\em On automatic homeomorphicity for transformation
    monoids}, Monatsh. Math. {\bf 179}:1 (2016), 129--148.

    \bibitem{PP1}
    C. Pech, M. Pech, {\em Reconstructing the Topology of the Elementary
    Self-embedding Monoids of Countable Saturated Structures}, Studia Logica
    {\bf 106}:3 (2018), 595--613.

    \bibitem{Per}
    J. P\'erez, C. Uzc\'ategui, {\em Topologies on the symmetric inverse
    semigroup}, Semigroup Forum {\bf 104}:2 (2022), 398--414.

    \bibitem{Pfi}
    H. Pfister, {\em Continuity of the inverse},
    Proc. Amer. Math. Soc. {\bf 95} (1985), 312--314.

    \bibitem{PS}
    M. Pinsker, S. Shelah, {\em Universality of the lattice of transformation
    monoids}, Proc. Amer. Math. Soc. {\bf 141}:9 (2013), 3005--3011.

    \bibitem{Preston}
    G. B. Preston, {\em Congruences on Brandt semigroups}, Mathematische
    Annalen {\bf 139} (1959), 91--94.

    \bibitem{Roma}
    S. Romaguera, M. Sanchis, {\em Continuity of the inverse in pseudocompact paratopological groups}, Algebra Colloquium
{\bf 14}:1 (2007), 167--175.

    \bibitem{Tka}
M.~Tkachenko, {\em Paratopological and Semitopological Groups Versus Topological Groups.} In: Hart K., van Mill J., Simon P. (eds) Recent Progress in General Topology III. Atlantis Press, Paris, 2014.
\end{thebibliography}
\end{document}